\documentclass[a4paper, 12pt]{amsart}

\usepackage{amsmath}
\usepackage{amssymb}
\usepackage{ascmac}
\usepackage{amsthm}

\makeatletter

\@addtoreset{equation}{section}
\makeatother 

\theoremstyle{plain}
\newtheorem{thm}{Theorem}[section]
\newtheorem*{thm*}{Theorem}
\newtheorem{prop}[thm]{Proposition}
\newtheorem{lem}[thm]{Lemma}
\newtheorem{cor}[thm]{Corollary}

\theoremstyle{definition}

\theoremstyle{remark}
\newtheorem{rem}[thm]{Remark}

\renewcommand{\epsilon}{\varepsilon}

\newcommand{\vol}{\operatorname{vol}}
\newcommand{\ric}{\operatorname{Ric}}
\newcommand{\Div}{\operatorname{div}}

\newcommand{\tr}{\operatorname{tr}}

\newcommand{\cut}{\mathrm{Cut}\,}

\usepackage{latexsym}

\title[Liouville theorems for harmonic map heat flow]{Liouville theorems for harmonic map heat flow along ancient super Ricci flow via reduced geometry}

\author{Keita Kunikawa}
\address{Cooperative Faculty of Education, Utsunomiya University, 350 Mine-Machi, Utsunomiya, 321-8505, Japan}
\email{kunikawa@cc.utsunomiya-u.ac.jp}

\author{Yohei Sakurai}
\address{Department of Mathematics, Saitama University, 255 Shimo-Okubo, Sakura-ku, Saitama-City, Saitama, 338-8570, Japan}
\email{ysakurai@rimath.saitama-u.ac.jp}

\date{July 31, 2021}

\begin{document}
\maketitle

\begin{abstract}
We study harmonic map heat flow along ancient super Ricci flow,
and derive several Liouville theorems with controlled growth from Perelman's reduced geometric viewpoint.
For non-positively curved target spaces, our growth condition is sharp. 
For positively curved target spaces,
our Liouville theorem is new even in the static case (i.e., for harmonic maps);
moreover, we point out that
the growth condition can be improved, and almost sharp in the static case.
This fills the gap between the Liouville theorem of Choi and the example constructed by Schoen-Uhlenbeck. 
\end{abstract}

\section{Background}

This is a continuation of \cite{KuS} on Liouville theorems for heat equation along ancient super Ricci flow.
The aim of this paper is to generalize the target spaces,
and formulate Liouville theorems for harmonic map heat flow.

\subsection{Ancient super Ricci flow}\label{sec:Ancient super Ricci flow}
A smooth manifold $(M,g(t))_{t\in I}$ with a time-dependent Riemannian metric is called \textit{Ricci flow} when
\begin{equation*}
\partial_{t}g= -2 \ric,
\end{equation*}
which has been introduced by Hamilton \cite{H1}.
A supersolution to this equation is called super Ricci flow.
Namely,
$(M,g(t))_{t\in I}$ is called \textit{super Ricci flow} if
\begin{equation*}\label{eq:sRF}
\partial_{t}g\geq -2 \ric,
\end{equation*}
which has been introduced by McCann-Topping \cite{MT} from the viewpoint of optimal transport theory.
Recently,
the super Ricci flow has begun to be investigated from various perspectives, especially metric measure geometry (see e.g., \cite{B1}, \cite{B2}, \cite{HN}, \cite{K}, \cite{KS1}, \cite{KS2}, \cite{LL1}, \cite{LL2}, \cite{LL3}, \cite{LL4}, \cite{LL5}, \cite{LL6}, \cite{S}).
A Ricci flow $(M,g(t))_{t\in I}$ is said to be \textit{ancient} when $I=(-\infty,0]$,
which is one of the crucial concepts in singular analysis of Ricci flow.
In the present paper,
we will focus on ancient super Ricci flow.

\subsection{Liouville theorems for ancient solutions to heat equation}\label{sec:heat eq}
The celebrated Yau's Liouville theorem states that
on a complete manifold of non-negative Ricci curvature,
any positive harmonic functions must be constant.
One of the natural research directions is to generalize his Liouville theorem for ancient solutions to heat equation
\begin{equation*}
\partial_t u=\Delta u.
\end{equation*}
Souplet-Zhang \cite{SZ} have proven the following parabolic analogue (see \cite[Theorem 1.2]{SZ}):
\begin{thm}[\cite{SZ}]\label{thm:SZ}
Let $(M,g)$ be a complete Riemannian manifold of non-negative Ricci curvature.
Then we have the following:
\begin{enumerate}\setlength{\itemsep}{+0.7mm}
\item Let $u:M\times (-\infty,0]\to (0,\infty)$ be a positive ancient solution to the heat equation.
If
\begin{equation*}
u(x,t)=\exp\left[o\left(d(x)+\sqrt{\vert t\vert}\right)\right]
\end{equation*}
near infinity, then $u$ must be constant.
Here $d(x)$ denotes the Riemannian distance from a fixed point; \label{enum:positive growth}
\item let $u:M\times (-\infty,0]\to \mathbb{R}$ be an ancient solution to the heat equation.
If
\begin{equation*}
u(x,t)=o\left(d(x)+\sqrt{\vert t\vert}\right)
\end{equation*}
near infinity, then $u$ is constant. \label{enum:usual growth}
\end{enumerate}
\end{thm}

The growth conditions in Theorem \ref{thm:SZ} are known to be sharp in the spatial direction (see \cite{SZ}, and cf. \cite{DD}). 
As mentioned in \cite[Section 1]{KuS},
one of the next research directions is the following:
For an ancient super Ricci flow $(M,g(t))_{t\in (-\infty,0]}$,
the problem is to find suitable growth conditions for a solution $u:M\times (-\infty,0]\to \mathbb{R}$ to heat equation such that $u$ must become constant.
In other words,
for the reverse time parameter
\begin{equation*}
\tau:=-t,
\end{equation*}
and for a backward super Ricci flow $(M,g(\tau))_{\tau \in [0,\infty)}$,
namely,
\begin{equation*}
\ric\geq \frac{1}{2}\partial_{\tau}g,
\end{equation*}
the problem is to find suitable conditions for a solution $u:M\times [0,\infty)\to \mathbb{R}$ to backward heat equation
\begin{equation*}
(\Delta+\partial_{\tau}) u=0
\end{equation*}
such that $u$ must become constant.
Guo-Philipowski-Thalmaier \cite{GPT2} have provided an approach to this problem from stochastic analytic viewpoint,
and obtained a Liouville theorem under a growth condition for entropy (see \cite[Theorem 2]{GPT2}).
On the other hand,
the authors \cite{KuS} have approached the problem from Perelman's reduced geometric viewpoint (cf. \cite{P}),
and established a Liouville theorem under a growth condition concerning reduced distance.

Now,
let us recall the precise statement of the Liouville theorem in \cite{KuS}.
To do so,
we fix some notations on a complete, time-dependent Riemannian manifold $(M,g(\tau))_{\tau \in [0,\infty)}$,
which is not necessarily backward super Ricci flow.
We put
\begin{equation*}\label{eq:hH}
h:=\frac{1}{2}\partial_{\tau}g,\quad H:=\tr h.
\end{equation*}

We begin with recalling the notion of reduced distance (more precisely, see Subsection \ref{sec:Preliminaries}).
For $(x,\tau)\in M\times (0,\infty)$,
let $L(x,\tau)$ be the \textit{$L$-distance} from a space-time base point $(x_0 ,0)$,
which is defined as the infimum of the so-called \textit{$\mathcal{L}$-length} over all curves $\gamma:[0,\tau]\to M$ with $\gamma(0)=x_0$ and $\gamma(\tau)=x$.
Then the \textit{reduced distance $\ell(x,\tau)$} is defined as
\begin{equation*}\label{eq:reduced distance}
\ell(x,\tau):=\frac{1}{2\sqrt{\tau}}L(x,\tau).
\end{equation*}
We say that
$(M,g(\tau))_{\tau \in [0,\infty)}$ is \textit{admissible} if
for every $\tau>0$ there is $c_{\tau}\geq 0$ depending only on $\tau$ such that $h \geq -c_{\tau} g$ on $[0,\tau]$.
The admissibility guarantees that
the $L$-distance is achieved by a minimal $\mathcal{L}$-geodesic.

Next,
for a (time-dependent) vector field $V$,
we recall the following \textit{M\"uller quantity} $\mathcal{D}(V)$ (see \cite[Definition 1.3]{M}),
and \textit{trace Harnack quantity} $\mathcal{H}(V)$ (see \cite{H2}, \cite[Definition 1.5]{M}):
\begin{align*}
\mathcal{D}(V)&:=-\partial_{\tau}H-\Delta H-2\Vert h \Vert^2+4\Div h(V)-2g(\nabla H,V)+2\ric(V,V)-2h(V,V),\\
\mathcal{H}(V)&:=-\partial_{\tau} H-\frac{H}{\tau}-2g(\nabla H,V)+2h(V,V).
\end{align*}

The main result in \cite{KuS} can be stated as follows (see \cite[Theorem 2.2]{KuS}):
\begin{thm}[\cite{KuS}]\label{thm:main result in KuS}
Let $(M,g(\tau))_{\tau \in [0,\infty)}$ be an admissible, complete backward super Ricci flow.
We assume
\begin{equation*}
\mathcal{D}(V)\geq 0,\quad \mathcal{H}(V) \geq -\frac{H}{\tau},\quad H\geq 0
\end{equation*}
for all vector fields $V$.
Then we have the following:
\begin{enumerate}\setlength{\itemsep}{+0.7mm}
\item Let $u:M\times [0,\infty)\to (0,\infty)$ be a positive solution to backward heat equation.
If
\begin{equation*}\label{eq:positive growth}
u(x,\tau)=\exp\left[o\left(\mathfrak{d}(x,\tau)+\sqrt{\tau}\right)\right]
\end{equation*}
near infinity, then $u$ is constant.
Here $\mathfrak{d}(x,\tau)$ is defined by
\begin{equation*}
\mathfrak{d}(x,\tau):=\sqrt{4\tau\,\ell(x,\tau)};
\end{equation*}
\item let $u:M\times [0,\infty)\to \mathbb{R}$ be a solution to backward heat equation.
If
\begin{equation*}\label{eq:usual growth}
u(x,\tau)=o\left(\mathfrak{d}(x,\tau)+\sqrt{\tau}\right)
\end{equation*}
near infinity, then $u$ is constant. 
\end{enumerate}
\end{thm}

In the static case of $h=0$,
Theorem \ref{thm:main result in KuS} is reduced to Theorem \ref{thm:SZ} (see \cite[Remark 2.3]{KuS}).

\section{Main results}

\subsection{Liouville theorems for ancient solutions to harmonic map heat flow}\label{sec:Liouville theorems for HMHF}
One can now consider the following problem:
For a backward super Ricci flow $(M,g(\tau))_{\tau \in [0,\infty)}$,
and a manifold $(N,\mathfrak{g})$ with an upper sectional curvature bound,
the problem is to find suitable conditions for a solution $u:M\times [0,\infty)\to N$ to backward harmonic map heat flow
\begin{equation}\label{eq:bHMHF}
(\Delta+\partial_{\tau}) u=0
\end{equation}
such that $u$ must be constant.
Here $\Delta$ is the tension field.
Guo-Philipowski-Thalmaier \cite{GPT2} have approached this problem from stochastic analytic viewpoint,
and produced various Liouville theorems (see \cite[Section 4]{GPT1}).
We here aim to approach the problem from Perelman's reduced geometric viewpoint.
Our first main result is the following Liouville theorem of Cheng type (cf. \cite{Che}):
\begin{thm}\label{thm:main1}
Let $(M,g(\tau))_{\tau \in [0,\infty)}$ be an admissible, complete backward super Ricci flow.
We assume
\begin{equation}\label{eq:main assumption}
\mathcal{D}(V)\geq 0,\quad \mathcal{H}(V) \geq -\frac{H}{\tau},\quad H\geq 0
\end{equation}
for all vector fields $V$.
Let $(N,\mathfrak{g})$ be a complete, simply connected Riemannian manifold with $\sec \leq 0$.
Let $u:M\times [0,\infty)\to N$ be a solution to backward harmonic map heat flow.
If
\begin{equation*}
\rho(u(x,\tau))=o\left(\mathfrak{d}(x,\tau)+\sqrt{\tau}\right)
\end{equation*}
near infinity, then $u$ is constant.
Here $\rho:N\to \mathbb{R}$ is the Riemannian distance function from a fixed point $y_0 \in N$.
\end{thm}

When $N=\mathbb{R}$,
Theorem \ref{thm:main1} is nothing but Theorem \ref{thm:main result in KuS}.
In the static case of $h=0$,
Theorem \ref{thm:main1} has been proved by Wang \cite{W} (see \cite[Theorem 1.3]{W}). Since growth conditions in these results are sharp in the spatial direction, so is the growth condition in Theorem \ref{thm:main1}. 

We also prove the following result for positively curved target spaces:
\begin{thm}\label{thm:main2}
Let $(M,g(\tau))_{\tau \in [0,\infty)}$ be an admissible, complete backward super Ricci flow.
We assume
\begin{equation}\label{eq:main assumption2}
\mathcal{D}(V)\geq 0,\quad \mathcal{H}(V) \geq -\frac{H}{\tau},\quad H\geq 0
\end{equation}
for all vector fields $V$.
Let $(N,\mathfrak{g})$ be a complete Riemannian manifold with $\sec \leq \kappa$ for $\kappa>0$.
Assume that an open geodesic ball $B_{\pi/2\sqrt{\kappa}}(y_0)$ of radius $\pi/2\sqrt{\kappa}$ centered at $y_0$ in $N$ does not meet the cut locus $\cut(y_0)$ of $y_0$.
Let $u:M\times [0,\infty)\to N$ be a solution to backward harmonic map heat flow.
If the image of $u$ is contained in $B_{\pi/2\sqrt{\kappa}}(y_0)$,
and if $u$ satisfies
\begin{equation*}\label{eq:growth main2}
\frac{1}{\cos \sqrt{\kappa}\rho(u(x,\tau))}=o\left(\mathfrak{d}(x,\tau)^{1/2}+\tau^{1/4}\right)
\end{equation*}
near infinity, then $u$ is constant.
\end{thm}

Theorems \ref{thm:main1} and \ref{thm:main2} follow from local gradient estimates (see Theorems \ref{thm:gradient1} and \ref{thm:gradient2}).

\subsection{Sharpness}
Let us discuss the sharpness concerning Theorem \ref{thm:main2}.
To do so,
we recall the Liouville theorem of Choi \cite{Cho} (see \cite[Theorem]{Cho}):
\begin{thm}[\cite{Cho}]\label{thm:choi}
Let $(M,g)$ be a complete Riemannian manifold of non-negative Ricci curvature,
and let $(N,\mathfrak{g})$ be a complete Riemannian manifold with $\sec\leq \kappa$ for $\kappa>0$. 
Let $u:M \to N$ be a harmonic map $($i.e., $\Delta u=0$$)$.
We assume that $B_{L}(y_0)$ is a regular $($i.e., $L\in (0, \pi/2\sqrt{\kappa})$ and $B_{L}(y_0)\cap \cut(y_0)=\emptyset$$)$, open geodesic ball.
If the image of $u$ is contained in $B_{L}(y_0)$,
then $u$ is constant. 
\end{thm}

Theorem \ref{thm:main2} enables us to improve Theorem \ref{thm:choi} as follows:
\begin{cor}\label{cor:follow}
Let $(M,g)$ be a complete Riemannian manifold of non-negative Ricci curvature,
and let $(N,\mathfrak{g})$ be a complete Riemannian manifold with $\sec\leq \kappa$ for $\kappa>0$. 
Assume that $B_{\pi/2\sqrt{\kappa}}(y_0)$ does not meet $\cut (y_0)$.
Let $u:M \to N$ be a harmonic map.
If the image of $u$ is contained in $B_{\pi/2\sqrt{\kappa}}(y_0)$,
and if $u$ satisfies a growth condition 
\begin{equation}\label{eq:follow}
\frac{1}{\cos\sqrt{\kappa}\rho(u(x))} = o(d(x)^{1/2})
\end{equation}
near infinity, then $u$ is constant. 
\end{cor}

The growth condition \eqref{eq:follow} controls the approach speed of $u$ to the boundary of $B_{\pi/2\sqrt{\kappa}}(y_0)$.
Note that if the image of $u$ is contained in a regular ball $B_{L}(y_0)$,
then the left hand side of \eqref{eq:follow} is bounded; in particular, \eqref{eq:follow} is trivially satisfied.

\begin{rem}\label{rem:new point}
In the literature of Liouville theorems for harmonic maps with positively curved targets,
the results in the form of Theorem \ref{thm:choi} have been examined (see e.g., \cite[Theorem 1]{HJW}, \cite[Theorem]{Cho}, \cite[Theorem 6.1]{Ke}, \cite[Theorem 1.4]{STW}, \cite[Example 3]{KuSt}, \cite[Theorem 3.2]{LW}, \cite[Theorem 2]{CJQ}, \cite[Theorem 2]{Q}, \cite[Corollary 1.8]{ZZZ}).
We emphasize that
in Corollary \ref{cor:follow}, such a condition is relaxed to 
a growth condition \eqref{eq:follow} beyond the traditional form.
\end{rem}

Although our formulation of Theorem \ref{thm:main2} and Corollary \ref{cor:follow} is new, the growth condition \eqref{eq:follow} is not sharp.
Actually, we can further improve it as follows:
\begin{thm}\label{thm:improve}
Let $(M,g)$ be a complete Riemannian manifold of non-negative Ricci curvature,
and let $(N, \mathfrak{g})$ be a complete Riemannian manifold with $\sec\leq \kappa$ for $\kappa>0$. 
Assume that $B_{\pi/2\sqrt{\kappa}}(y_0)$ does not meet $\cut (y_0)$.
Let $u:M \to N$ be a harmonic map.
If the image of $u$ is contained in $B_{\pi/2\sqrt{\kappa}}(y_0)$,
and if $u$ satisfies a growth condition 
\begin{equation}\label{eq:improve}
\frac{1}{\cos\sqrt{\kappa}\rho(u(x))} = o(d(x))
\end{equation}
near infinity, then $u$ is constant. 
\end{thm}

We can obtain Theorem \ref{thm:improve} by adopting the technique for minimal hypersurfaces developed by Ecker-Huisken \cite{EH}.

\begin{rem}
According to Schoen-Uhlenbeck \cite{SU} (see also \cite{GS}),
a harmonic map $u:\mathbb{R}^m \to \mathbb{S}^n_+$ is necessarily constant for $m\leq 6$,
and for $m\geq 7$ such a map exists as a radial solution,
where $\mathbb{S}^n_+$ is the open hemisphere.
In Subsection \ref{subsec:SU-example}, we observe that
the growth of the radial solution is greater than the linear order.
Moreover,
it approaches the linear order as $m\to \infty$.
In this sense,
our growth condition \eqref{eq:improve} is almost sharp. 
\end{rem}

\section{Preliminaries}\label{sec:Preliminaries}
We review some facts on Perelman's reduced geometry.
The references are \cite{CCGG}, \cite{M}, \cite{P}, \cite{Ye}, \cite{Y1}, \cite{Y2}, \cite{KuS}.
We mainly refer to \cite[Section 3]{KuS}.
Throughout this subsection,
let $(M,g(\tau))_{\tau \in [0,\infty)}$ be an $m$-dimensional, complete time-dependent Riemannian manifold.

For a curve $\gamma:[\tau_1,\tau_2]\to M$,
its \textit{$\mathcal{L}$-length} is defined as
\begin{equation*}
\mathcal{L}(\gamma):=\int^{\tau_2}_{\tau_1}\sqrt{\tau}\left( H +\left\Vert \frac{d\gamma}{d\tau} \right\Vert^{2} \right)\,d\tau.
\end{equation*}
It is well-known that
its critical point over all curves with fixed endpoints is characterized by the following \textit{$\mathcal{L}$-geodesic equation}:
\begin{equation*}\label{geod}
X:=\frac{d\gamma}{d\tau},\quad \nabla_X X - \frac{1}{2}\nabla H+ \frac{1}{2\tau}X + 2h(X) =0.
\end{equation*}
For $(x,\tau)\in M\times (0,\infty)$,
the \textit{$L$-distance} $L(x,\tau)$ and \textit{reduced distance} $\ell(x,\tau)$ from a space-time base point $(x_0,0)$ are defined by
\begin{equation}\label{eq:L and l}
L(x,\tau):=\inf_{\gamma}\mathcal{L}(\gamma),\quad \ell(x,\tau):=\frac{1}{2\sqrt{\tau}}L(x,\tau),
\end{equation}
where the infimum is taken over all curves $\gamma:[0,\tau]\to M$ with $\gamma(0)=x_0$ and $\gamma(\tau)=x$.
A curve is called \textit{minimal $\mathcal{L}$-geodesic} from $(x_0,0)$ to $(x,\tau)$ if it attains the infimum of (\ref{eq:L and l}).
We also set
\begin{equation*}
\overline{L}(x,\tau):=4\tau \,\ell(x,\tau).
\end{equation*}

We now assume that
$(M,g(\tau))_{\tau \in [0,\infty)}$ is admissible (see Subsection \ref{sec:heat eq}).
In this case,
for every $(x,\tau)\in M\times (0,\infty)$,
there exists at least one minimal $\mathcal{L}$-geodesic.
Also,
the functions $L(\cdot,\tau)$ and $L(x,\cdot)$ are locally Lipschitz in $(M,g(\tau))$ and $(0,\infty)$,
respectively;
in particular,
they are differentiable almost everywhere.

Assume that
$\ell$ is smooth at $(\overline{x},\overline{\tau})\in M\times (0,\infty)$.
We have (see \cite[Lemmas 3.5 and 3.6]{KuS}):
\begin{lem}[\cite{KuS}]\label{lem:reduced heat estimate2}
Let $K\geq 0$.
We assume
\begin{equation*}\label{eq:assume reduced heat estimate2}
\mathcal{D}(V) \geq -2K\left(H + \Vert V \Vert^2\right),\quad H\geq 0
\end{equation*}
for all vector fields $V$.
Then at $(\overline{x},\overline{\tau})$ we have
\begin{equation*}
(\Delta+\partial_\tau)\overline{L} \leq 2m+2K \overline{L}.
\end{equation*}
\end{lem}

\begin{lem}[\cite{KuS}]\label{lem:reduced gradient estimate}
We assume
\begin{equation*}\label{eq:assumption rge}
\mathcal{H}(V) \geq -\frac{H}{\tau},\quad H\geq 0
\end{equation*}
for all vector fields $V$.
Then at $(\overline{x},\overline{\tau})$ we have
\begin{equation*}
\Vert \nabla \mathfrak{d} \Vert^2\leq 3.
\end{equation*}
\end{lem}

\section{Proof of Theorem \ref{thm:main1}}\label{sec:Proof of main results}
In this section,
we prove Theorem \ref{thm:main1}.
For $K\in \mathbb{R}$,
a time-dependent Riemannian manifold $(M,g(t))_{t\in I}$ is called \textit{$K$-super Ricci flow} if
\begin{equation*}
\frac{1}{2}\partial_{t}g+\ric \geq K g.
\end{equation*}
The key ingredient is the following local gradient estimate (cf. \cite[Theorem 2.8]{KuS}):
\begin{thm}\label{thm:gradient1}
For $K\geq 0$,
let $(M,g(\tau))_{\tau \in [0,\infty)}$ be an $m$-dimensional, admissible, complete backward $(-K)$-super Ricci flow,
namely,
\begin{equation*}\label{eq:KsRF}
\ric \geq h-K g.
\end{equation*}
We assume 
\begin{equation}\label{eq:Kmain assumption}
\mathcal{D}(V)\geq -2K\left(H+\Vert V\Vert^2   \right),\quad \mathcal{H}(V) \geq -\frac{H}{\tau},\quad H\geq 0
\end{equation}
for all vector fields $V$.
Let $(N,\mathfrak{g})$ stand for a complete, simply connected Riemannian manifold with $\sec \leq 0$.
For a fixed $y_0\in N$,
let $\rho:N\to \mathbb{R}$ be the Riemannian distance function from $y_0$.
Let $u:M\times [0,\infty)\to N$ be a solution to backward harmonic map heat flow.
For $R,T>0$ and $A>0$,
we suppose $2\rho \circ u\leq A$ on
\begin{equation*}
Q_{R,T}:=\left\{\,(x,\tau)\in M\times \left(0,T\right] \,\,\, \middle|\,\,\, \mathfrak{d}(x,\tau)\leq R  \,\right\}.
\end{equation*}
Then there exists a positive constant $C_{m}>0$ depending only on $m$ such that on $Q_{R/2,T/4}$,
\begin{equation*}
\frac{\Vert du \Vert}{A^2-\rho^2 \circ u}\leq \frac{C_m}{A} \left( \frac{1}{R}+\frac{1}{\sqrt{T}}+\sqrt{K}  \right).
\end{equation*}
\end{thm}

In the static case of $h=0$,
Wang \cite{W} has obtained Theorem \ref{thm:gradient1} (see \cite[Theorem 1.2]{W}).
We will prove Theorem \ref{thm:gradient1} along the line of the proof of \cite[Theorem 1.2]{W}.

\subsection{Backward harmonic map heat flows}\label{sec:Formulas for backward heat equations}
In this and next section,
let $(M,g(\tau))_{\tau \in [0,\infty)}$ denote an $m$-dimensional, admissible, complete time-dependent Riemannian manifold,
and let $(N,\mathfrak{g})$ be a complete Riemannian manifold.
Moreover,
for a fixed $y_0\in N$,
let $\rho:N\to \mathbb{R}$ stand for the Riemannian distance function from $y_0$.
We study properties of a solution $u:M\times [0,\infty)\to N$ to backward harmonic map heat flow.
We start with the following:
\begin{lem}\label{lem:easy}
\begin{align*}
(\Delta+\partial_{\tau}) \Vert du\Vert^2 =2\Vert \nabla du\Vert^2&+2\sum^{m}_{i=1}\mathfrak{g}(du(\mathcal{R}(e_i)),du(e_i))\\
                                                 &-2\sum^{m}_{i,j=1}\mathfrak{g}(R(du(e_i),du(e_j))du(e_j),du(e_i)),
\end{align*}
where $\mathcal{R}:=\ric-h$,
and $\{e_i\}^{m}_{i=1}$ is an orthonormal frame on $M$ at some fixed time.
\end{lem}
\begin{proof}
By direct computations and backward harmonic map heat flow equation \eqref{eq:bHMHF},
we have the following (cf. \cite[Lemma 4.5]{AH}):
\begin{align*}
\partial_{\tau}\Vert du \Vert^2
&=-\sum^{m}_{i=1}\mathfrak{g}(du((\partial_{\tau}g)(e_i)),du(e_i))+2\sum^{m}_{i=1}\mathfrak{g}(\nabla^{u^{-1}TN}_{e_i}(\partial_{\tau}u),du(e_i) )\\
&=-2\sum^{m}_{i=1}\mathfrak{g}(du(h(e_i)),du(e_i))-2\sum^{m}_{i=1}\mathfrak{g}(\nabla^{u^{-1}TN}_{e_i}\Delta u,du(e_i) )\\
&=-2\sum^{m}_{i=1}\mathfrak{g}(du(\ric(e_i)),du(e_i))-2\sum^{m}_{i=1}\mathfrak{g}(\nabla^{u^{-1}TN}_{e_i}\Delta u,du(e_i) )\\
&\qquad \qquad \qquad \qquad \qquad \qquad \quad \,\,\,\,  +2\sum^{m}_{i=1}\mathfrak{g}(du(\mathcal{R}(e_i)),du(e_i)),
\end{align*}
here $u^{-1}TN$ denotes the induced vector bundle from $TN$ by $u$,
and $\nabla^{u^{-1}TN}$ is the canonical connection over $u^{-1}TN$.
Combining the above equation and the Bochner formula of Eells-Sampson type tells us the following (see e.g., \cite[Remark 1.15]{AMR}):
\begin{align*}
\frac{1}{2}\Delta \Vert du\Vert^2 &=\Vert \nabla du\Vert^2+\sum^{m}_{i=1}\mathfrak{g}(\nabla^{u^{-1}TN}_{e_i}\Delta u,du(e_i) )+\sum^{m}_{i=1}\mathfrak{g}(du(\ric(e_i)),du(e_i))\\
                                                 &\qquad \qquad \,\,\, \, -\sum^{m}_{i,j=1}\mathfrak{g}(R(du(e_i),du(e_j))du(e_j),du(e_i))\\
                                                 &=\Vert \nabla du\Vert^2+\sum^{m}_{i=1}\mathfrak{g}(du(\mathcal{R}(e_i)),du(e_i))-\frac{1}{2}\partial_{\tau}\Vert du \Vert^2\\
                                                 &\qquad \qquad \,\,\, \,-\sum^{m}_{i,j=1}\mathfrak{g}(R(du(e_i),du(e_j))du(e_j),du(e_i)).
\end{align*}
This completes the proof.
\end{proof}

We next show the following:
\begin{lem}\label{lem:simple}
Let $(N,\mathfrak{g})$ be simply connected, and $\sec \leq 0$.
For $A>0$, we assume $2\rho\circ u\leq A$.
Set
\begin{equation}\label{eq:key function}
w:=\frac{\Vert du\Vert^2}{(A^2-\rho^2 \circ u)^2}.
\end{equation}
Then we have
\begin{align*}
(\Delta+\partial_{\tau}) w-2\frac{g(\nabla w,\nabla (\rho^2\circ u))}{A^2-\rho^2 \circ u}&\geq 4(A^2-\rho^2 \circ u)w^2\\
      &\quad +\frac{2}{(A^2-\rho^2 \circ u)^2}\sum^{m}_{i=1}\mathfrak{g}(du(\mathcal{R}(e_i)),du(e_i)).
\end{align*}
\end{lem}
\begin{proof}
By straightforward computations,
\begin{align*}
\nabla w&=\frac{\nabla \Vert du\Vert^2}{(A^2-\rho^2 \circ u)^2}+2\frac{\Vert du\Vert^2~\nabla (\rho^2 \circ u)}{(A^2-\rho^2 \circ u)^3},\\
\Delta w&=\frac{\Delta \Vert du\Vert^2}{(A^2-\rho^2 \circ u)^2}+\frac{4g(\nabla \Vert du\Vert^2,\nabla (\rho^2 \circ u))}{(A^2-\rho^2 \circ u)^3}+\frac{2\Vert du\Vert^2~ \Delta (\rho^2 \circ u)}{(A^2-\rho^2 \circ u)^3}+\frac{6\Vert \nabla (\rho^2 \circ u) \Vert^2~\Vert du\Vert^2}{(A^2-\rho^2 \circ u)^4}\\
                &=\frac{2g(\nabla w,\nabla (\rho^2\circ u))}{A^2-\rho^2 \circ u}+\frac{2\Vert du\Vert^2~ \Delta (\rho^2 \circ u)}{(A^2-\rho^2 \circ u)^3}+\frac{\Delta \Vert du\Vert^2}{(A^2-\rho^2 \circ u)^2}\\
                &\quad +\frac{2g(\nabla \Vert du\Vert^2,\nabla (\rho^2 \circ u))}{(A^2-\rho^2 \circ u)^3}+\frac{2\Vert \nabla (\rho^2 \circ u) \Vert^2~\Vert du\Vert^2}{(A^2-\rho^2 \circ u)^4},\\
\partial_{\tau} w&=\frac{\partial_{\tau} \Vert du\Vert^2}{(A^2-\rho^2 \circ u)^2}+\frac{2\Vert du\Vert^2~\partial_{\tau} (\rho^2 \circ u)}{(A^2-\rho^2 \circ u)^3}.
\end{align*}
It follows that
\begin{align*}
(\Delta+\partial_{\tau}) w-\frac{2g(\nabla w,\nabla (\rho^2\circ u))}{A^2-\rho^2 \circ u}&=\frac{2\Vert du\Vert^2~ (\Delta+\partial_{\tau}) (\rho^2 \circ u)}{(A^2-\rho^2 \circ u)^3}+\frac{(\Delta+\partial_{\tau}) \Vert du\Vert^2}{(A^2-\rho^2 \circ u)^2}\\
                &\quad +\frac{2g(\nabla \Vert du\Vert^2,\nabla (\rho^2 \circ u))}{(A^2-\rho^2 \circ u)^3}+\frac{2\Vert \nabla (\rho^2 \circ u) \Vert^2~\Vert du\Vert^2}{(A^2-\rho^2 \circ u)^4}.
\end{align*}
Since $(N,\mathfrak{g})$ is simply connected and $\sec \leq 0$,
the Greene-Wu Hessian comparison yields the following (see e.g., \cite[(1.263)]{AMR}, and also \cite[(1.181)]{AMR}):
\begin{equation*}
(\Delta+\partial_{\tau}) (\rho^2 \circ u)=\sum^{m}_{i=1}\nabla^2 \rho^2(du(e_i),du(e_i))\geq 2\Vert du\Vert^2,
\end{equation*}
where we also used the backward harmonic map heat flow equation \eqref{eq:bHMHF}.
Furthermore,
in view of Lemma \ref{lem:easy} and $\sec \leq 0$,
\begin{equation*}
(\Delta+\partial_{\tau}) \Vert du\Vert^2 \geq 2\Vert \nabla du\Vert^2+2\sum^{m}_{i=1}\mathfrak{g}(du(\mathcal{R}(e_i)),du(e_i)).
\end{equation*}
Combining the above estimates,
we see
\begin{align*}
(\Delta+\partial_{\tau}) w-\frac{2g(\nabla w,\nabla (\rho^2\circ u))}{A^2-\rho^2 \circ u}&\geq 4(A^2-\rho^2 \circ u)w^2\\
      &\quad+\frac{2}{(A^2-\rho^2 \circ u)^2}\sum^{m}_{i=1}\mathfrak{g}(du(\mathcal{R}(e_i)),du(e_i))+2\mathcal{F},
\end{align*}
where
\begin{equation*}
\mathcal{F}:=\frac{\Vert \nabla du\Vert^2}{(A^2-\rho^2 \circ u)^2}+\frac{\Vert \nabla (\rho^2 \circ u) \Vert^2~\Vert du\Vert^2}{(A^2-\rho^2 \circ u)^4}+\frac{g(\nabla \Vert du\Vert^2,\nabla (\rho^2 \circ u))}{(A^2-\rho^2 \circ u)^3}.
\end{equation*}

Now,
it suffices to check that
$\mathcal{F}$ is non-negative.
For the first two terms,
the inequality of arithmetic-geometric means, and the Kato inequality imply
\begin{align*}
\frac{\Vert \nabla du\Vert^2}{(A^2-\rho^2 \circ u)^2}+\frac{\Vert \nabla (\rho^2 \circ u) \Vert^2~\Vert du\Vert^2}{(A^2-\rho^2 \circ u)^4}
&\geq \frac{2\Vert \nabla du\Vert \Vert \nabla (\rho^2 \circ u) \Vert \Vert du\Vert }{(A^2-\rho^2 \circ u)^3}\\
&\geq \frac{\Vert \nabla \Vert du\Vert^2 \Vert~\Vert \nabla (\rho^2 \circ u)\Vert }{(A^2-\rho^2 \circ u)^3}.
\end{align*}
The Cauchy-Schwarz inequality tells us the desired conclusion.
\end{proof}

\subsection{Cut-off arguments}\label{sec:Cut-off arguments}
Let us recall the following elementary fact:
\begin{lem}\label{lem:cutoff}
Let $R,T>0,\,\alpha \in (0,1)$.
Then there is a smooth function $\psi:[0,\infty)\times [0,\infty)\to [0,1]$ which is supported on $[0,R]\times [0,T]$,
and a constant $C_\alpha>0$ depending only on $\alpha$ such that the following hold:
\begin{enumerate}\setlength{\itemsep}{3pt}
\item $\psi\equiv 1$ on $[0,R/2]\times [0,T/4]$;
\item $\partial_r \psi \leq 0$ on $[0,\infty)\times [0,\infty)$, and $\partial_r \psi \equiv 0$ on $[0,R/2]\times [0,\infty)$;
\item we have
\begin{equation*}
         \frac{\vert \partial_r \psi \vert}{\psi^\alpha}\leq \frac{C_{\alpha}}{R},\quad \frac{\vert \partial^2_{r}\psi\vert}{\psi^\alpha}\leq \frac{C_{\alpha}}{R^2},\quad \frac{\vert \partial_{\tau} \psi\vert}{\psi^{1/2}}\leq \frac{C}{T},
\end{equation*}
where $C>0$ is a universal constant.
\end{enumerate}
\end{lem}

We deduce the following:
\begin{prop}\label{prop:difficult}
Let $K\geq 0$.
We assume
\begin{equation*}\label{eq:cutoff assumption}
\mathcal{R}(V)\geq -K\Vert V \Vert^2,\quad \mathcal{D}(V)\geq -2K\left(H+\Vert V \Vert^2 \right),\quad \mathcal{H}(V) \geq -\frac{H}{\tau},\quad H\geq 0
\end{equation*}
for all vector fields $V$.
Let $(N,\mathfrak{g})$ be simply connected, and $\sec \leq 0$.
Let $u:M\times [0,\infty)\to N$ be a solution to backward harmonic map heat flow.
For $R,T>0$ and $A>0$,
we suppose $2\rho \circ u\leq A$ on $Q_{R,T}$.
We define $w$ as \eqref{eq:key function} on $Q_{R,T}$.
We also take a function $\psi:[0,\infty)\times [0,\infty)\to [0,1]$ in Lemma \ref{lem:cutoff} with $\alpha=3/4$,
and set
\begin{equation}\label{eq:cutoff function}
\psi(x,\tau):=\psi(\mathfrak{d}(x,\tau),\tau).
\end{equation}
Then we have
\begin{equation*}
(\psi w)^2\leq \frac{1}{A^4}\left(\frac{\overline{C}_m}{R^4}+\frac{\widetilde{C}_1}{T^2}+\widetilde{C}_2 K^2\right)+\frac{1}{A^2}\Phi.
\end{equation*}
at every point in $Q_{R,T}$ such that the reduced distance is smooth,
where for the universal constants $C_{3/4},C>0$ given in Lemma \ref{lem:cutoff},
we put
\begin{align}\label{eq:dimension epsilon}
\overline{C}_{m}&:=6C^2_{3/4}\left(m^2+\frac{9}{4}+\frac{369}{32}C^2_{3/4}\right),\quad \widetilde{C}_1:=\frac{3}{2}C^2,\quad \widetilde{C}_2:=6\left(1+\frac{C^2_{3/4}}{4}\right), \\  \label{eq:maximal constant}
\Phi&:=(\Delta+\partial_{\tau})(\psi w)-\frac{2g\left( \nabla\psi, \nabla(\psi w) \right)}{\psi}-\frac{2g(\nabla(\psi w),\nabla (\rho^2 \circ u))}{A^2-\rho^2\circ u}.
\end{align}
\end{prop}
\begin{proof}
In virtue of Lemma \ref{lem:simple},
\begin{align*}
\Phi&=\psi \left(\Delta+\partial_{\tau} \right)w-\frac{2\psi g(\nabla w,\nabla (\rho^2\circ u))}{A^2-\rho^2 \circ u}
+w\left(\Delta+\partial_{\tau}\right)\psi-\frac{2w\Vert \nabla \psi \Vert^2}{\psi}-\frac{2w g(\nabla \psi,\nabla (\rho^2\circ u))}{A^2-\rho^2 \circ u}\\
    &\geq 4(A^2-\rho^2 \circ u)\psi w^2+\frac{2\psi }{(A^2-\rho^2 \circ u)^2}\sum^{m}_{i=1}\mathfrak{g}(du(\mathcal{R}(e_i)),du(e_i))
    +w\left(\Delta+\partial_{\tau} \right)\psi-\frac{2w\Vert \nabla \psi \Vert^2}{\psi}\\
    &\qquad \qquad \qquad \qquad \qquad \qquad \qquad \qquad \qquad \qquad \qquad \qquad\qquad \quad \, -2\frac{w g(\nabla \psi,\nabla (\rho^2\circ u))}{A^2-\rho^2 \circ u}.
\end{align*}
It follows that
\begin{equation}\label{eq:main difficult}
4(A^2-\rho^2 \circ u)\psi w^2 \leq \Psi_1+\Psi_2+\Psi_3+\Psi_4+\Phi
\end{equation}
for
\begin{align*}
\Psi_1&:=-\frac{2\psi }{(A^2-\rho^2 \circ u)^2}\sum^{m}_{i=1}\mathfrak{g}(du(\mathcal{R}(e_i)),du(e_i)),\quad \Psi_2:=-w \left(\Delta+\partial_{\tau}\right)\psi,\\ \Psi_3&:=\frac{2w\Vert \nabla \psi \Vert^2}{\psi},\quad \Psi_4:=\frac{2w g(\nabla \psi,\nabla (\rho^2\circ u))}{A^2-\rho^2 \circ u}.
\end{align*}

We provide upper bounds of $\Psi_1, \Psi_2, \Psi_3,\Psi_4$.
The following Young inequality plays a crucial role:
For all $p,q\in (1,\infty)$ with $p^{-1}+q^{-1}=1$, $a,b\geq 0$, and $\epsilon>0$,
\begin{equation}\label{eq:Young}
ab\leq \frac{\epsilon a^p}{p}+\frac{b^q}{\epsilon^{q/p} q}.
\end{equation} 
The inequality
\begin{equation}\label{eq:grad est cutoff}
\frac{\Vert \nabla \psi \Vert^2}{\psi^{3/2}}\leq \frac{3 C^2_{3/4}}{R^2}
\end{equation}
is also useful,
which follows from Lemmas \ref{lem:reduced gradient estimate} and \ref{lem:cutoff}.
We first study an upper bound of $\Psi_1$.
By the assumption for $\mathcal{R}(V)$,
the Young inequality \eqref{eq:Young} with $p,q=2$,
and $\psi\leq 1$,
\begin{equation}\label{eq:estimate of C_1}
\Psi_1=-\frac{2\psi }{(A^2-\rho^2 \circ u)^2}\sum^{m}_{i=1}\mathfrak{g}(du(\mathcal{R}(e_i)),du(e_i)) \leq 2K \psi w\leq \epsilon \psi^2 w^2+\frac{K^2}{\epsilon}\leq \epsilon \psi w^2+\frac{K^2}{\epsilon}.
\end{equation}
We next produce an upper bound of $\Psi_2$.
We see
\begin{align*}
\Psi_2&=-w \left(\Delta+\partial_{\tau}\right)\psi=-w  \left(           \partial_r \psi(\Delta+\partial_{\tau})\mathfrak{d}+\partial^2_r \psi \Vert \nabla \mathfrak{d} \Vert^2+\partial_{\tau} \psi \right)\\ \notag
                  &=-w  \left[           \partial_r \psi   \left(  \frac{1}{2\mathfrak{d}}(\Delta+\partial_{\tau})\overline{L}-\frac{\Vert \nabla \overline{L} \Vert^2}{4\mathfrak{d}^3}  \right)+\partial^2_r \psi \Vert \nabla \mathfrak{d} \Vert^2+\partial_{\tau} \psi \right]\\ \notag
                  &=\frac{w \vert \partial_r \psi \vert}{2\mathfrak{d}}(\Delta+\partial_{\tau})\overline{L}-w \vert \partial_r \psi \vert \frac{\Vert \nabla \overline{L} \Vert^2}{4\mathfrak{d}^3}-w \,\partial^2_r \psi \Vert \nabla \mathfrak{d} \Vert^2-w\,\partial_{\tau} \psi\\ \notag
                  & \leq \frac{w \vert \partial_r \psi \vert}{2\mathfrak{d}}(\Delta+\partial_{\tau})\overline{L}+w \vert \partial^2_r \psi\vert \Vert \nabla \mathfrak{d} \Vert^2+w \,\vert \partial_{\tau} \psi \vert.
\end{align*}
Lemmas \ref{lem:reduced heat estimate2}, \ref{lem:reduced gradient estimate} and $\overline{L}=\mathfrak{d}^2$ yield
\begin{align*}
\Psi_2&\leq m \frac{w \vert \partial_r \psi \vert}{\mathfrak{d}}+K w \vert \partial_r \psi \vert \mathfrak{d}+3w \vert \partial^2_r \psi \vert+w \, \vert \partial_{\tau} \psi \vert \\ \notag
                       &\leq \frac{2m}{R} w \vert \partial_r \psi \vert+KR w \vert \partial_r \psi \vert+3w \vert \partial^2_r \psi \vert+w \,\vert \partial_{\tau} \psi \vert,
\end{align*}
where in the second inequality,
we used the fact that
$\partial_r \psi$ vanishes on $[0,R/2]\times [0,\infty)$.
From the Young inequality \eqref{eq:Young} with $p,q=2$,
Lemma \ref{lem:cutoff},
and $\psi\leq 1$,
we derive
\begin{align}\label{eq:estimate of A}
\Psi_2&\leq \left( \epsilon  \psi w^2+\frac{m^2}{R^2}\frac{ \vert \partial_r \psi \vert^2}{\epsilon  \psi}   \right)
+\left( \epsilon \psi w^2+\frac{K^2 R^2}{4}\frac{ \vert \partial_r \psi \vert^2}{\epsilon \psi}   \right)\\ \notag
&\qquad \qquad \qquad \qquad \qquad \,\,\,+\left( \epsilon \psi w^2+\frac{9}{4}\frac{ \vert \partial^2_r \psi \vert^2}{\epsilon \psi}   \right)
+\left( \epsilon \psi w^2+\frac{1}{4}\frac{ \vert \partial_{\tau}  \psi\vert^2}{\epsilon \psi}   \right) \\ \notag
&\leq 4\epsilon  \psi w^2+\frac{C^2_{3/4}}{\epsilon}\left(m^2+\frac{9}{4}\right)\frac{ \psi^{1/2}}{R^4}+\frac{C^2}{4\epsilon}\frac{1}{T^2}+\frac{C^2_{3/4}}{4\epsilon}K^2\psi^{1/2}\\ \notag
&\leq 4\epsilon  \psi w^2+\frac{C^2_{3/4}}{\epsilon}\left(m^2+\frac{9}{4}\right)\frac{1}{R^4}+\frac{C^2}{4\epsilon}\frac{1}{T^2}+\frac{C^2_{3/4}}{4\epsilon}K^2.
\end{align}
We give an upper bound of $\Psi_3$.
By the Young inequality \eqref{eq:Young} with $p,q=2$,
and \eqref{eq:grad est cutoff},
\begin{equation}\label{eq:estimate of B}
\Psi_3= \frac{2w\Vert \nabla \psi \Vert^2}{\psi} \leq \epsilon \psi w^2+\frac{\Vert \nabla \psi \Vert^4}{\epsilon \psi^3}\leq \epsilon \psi w^2+\frac{9C^4_{3/4}}{\epsilon}\frac{1}{R^4}.
\end{equation}
We finally examine $\Psi_4$.
The Cauchy-Schwarz inequality,
the Young inequality (\ref{eq:Young}) with $p=4/3,q=4,\epsilon=4/3$,
and \eqref{eq:grad est cutoff} lead us to
\begin{align}\label{eq:estimate of C}
\Psi_4&=2\frac{w g(\nabla \psi,\nabla (\rho^2\circ u))}{A^2-\rho^2 \circ u}\leq \frac{2w\Vert \nabla \psi \Vert \Vert \nabla (\rho^2 \circ u) \Vert}{A^2-\rho^2 \circ u}\leq 2 A w^{3/2} \Vert \nabla \psi \Vert\\ \notag
                  &\leq A^2\psi w^2+\frac{27}{16} \frac{1}{A^2}  \frac{\Vert \nabla \psi \Vert^4}{\psi^3}\leq A^2\psi w^2+\frac{243C^4_{3/4}}{16} \,\frac{1}{A^2}\frac{1}{R^4}.
\end{align}

By summarizing (\ref{eq:main difficult}), (\ref{eq:estimate of C_1}), (\ref{eq:estimate of A}), (\ref{eq:estimate of B}), (\ref{eq:estimate of C}),
\begin{align*}
3A^2\psi w^2 &\leq 4(A^2-\rho^2 \circ u)\psi w^2\\
&\leq (6\epsilon+A^2)  \psi w^2+\frac{C^2_{3/4}}{\epsilon}\left(m^2+\frac{9}{4}+9C^2_{3/4}+  \frac{243 \epsilon C^2_{3/4}}{16} \,\frac{1}{A^2}   \right)\frac{1}{R^4}\\
&\qquad \qquad \qquad \quad \, \,+\frac{C^2}{4\epsilon}\frac{1}{T^2}+\frac{1}{\epsilon}\left(1+\frac{C^2_{3/4}}{4}\right)K^2+\Phi.
\end{align*}
 Letting $\epsilon \to A^2/6$,
we have
\begin{equation*}
\psi w^2 \leq \frac{1}{A^4}\left(\frac{\overline{C}_m}{R^4}+\frac{\widetilde{C}_1}{T^2}+\widetilde{C}_2 K^2\right)+\frac{1}{A^2}\Phi.
\end{equation*}
Since $(\psi w)^2\leq \psi w^2$,
we arrive at the desired inequality.
\end{proof}

\subsection{Proof of Theorems \ref{thm:main1} and \ref{thm:gradient1}}\label{sec:Proof of main theorems}

Let us conclude Theorem \ref{thm:gradient1}.
\begin{proof}[Proof of Theorem \ref{thm:gradient1}]
For $K\geq 0$,
let $(M,g(\tau))_{\tau \in [0,\infty)}$ be backward $(-K)$-super Ricci flow satisfying \eqref{eq:Kmain assumption} for all vector fields $V$.
Let $(N,\mathfrak{g})$ be simply connected, and $\sec \leq 0$.
Let $u:M\times [0,\infty)\to N$ be a solution to backward harmonic map heat flow.
For $R,T>0$ and $A>0$,
we suppose $2\rho \circ u\leq A$ on $Q_{R,T}$.
We define functions $w$ and $\psi$ as \eqref{eq:key function} and \eqref{eq:cutoff function},
respectively.
For $\theta>0$
we define a compact subset $Q_{R,T,\theta}$ of $Q_{R,T}$ by
\begin{equation}\label{eq:compact}
Q_{R,T,\theta}:=\{(x,\tau)\in Q_{R,T} \mid \tau\in [\theta,T]\}.
\end{equation}
Fix a small $\theta\in (0,T/4)$,
and take a maximum point $(\overline{x},\overline{\tau})$ of $\psi w$ in $Q_{R,T,\theta}$.
By virtue of the Calabi argument,
we may assume that
the reduced distance is smooth at $(\overline{x},\overline{\tau})$ (cf. \cite[Remark 3.3]{KuS}).
Using Proposition \ref{prop:difficult},
we see
\begin{equation*}
(\psi w)^2\leq \frac{c_m}{A^4} \left(\frac{1}{R^4}+\frac{1}{T^2}+K^2 \right)+\frac{1}{A^2}\Phi
\end{equation*}
at $(\overline{x},\overline{\tau})$ for
\begin{equation*}
c_{m}:=\max \left\{\overline{C}_{m},\widetilde{C}_1,\widetilde{C}_2 \right\},
\end{equation*}
where $\overline{C}_{m},\widetilde{C}_1,\widetilde{C}_2>0$ and $\Phi$ are defined as $(\ref{eq:dimension epsilon})$ and (\ref{eq:maximal constant}),
respectively.
On the other hand,
since $(\overline{x},\overline{\tau})$ is a maximum point,
\begin{equation*}
\Delta(\psi w)\leq 0,\quad \partial_{\tau} (\psi w)\leq 0,\quad \nabla (\psi w)=0
\end{equation*}
at $(\overline{x},\overline{\tau})$;
in particular, $\Phi(\overline{x},\overline{\tau})\leq 0$.
Therefore,
\begin{equation*}\label{eq:smooth conclusion}
(\psi w)(x,\tau) \leq (\psi w)(\overline{x},\overline{\tau})\leq \frac{c^{1/2}_{m}}{A^2}  \left(\frac{1}{R^4}+\frac{1}{T^2}+K^2     \right)^{1/2}\leq \frac{c^{1/2}_{m}}{A^2} \left(\frac{1}{R^2}+\frac{1}{T}+K\right)
\end{equation*}
for all $(x,\tau)\in Q_{R,T,\theta}$.
By $\psi \equiv 1$ on $Q_{R/2,T/4,\theta}$,
and by the definition of $w$,
\begin{equation*}
\frac{\Vert d u \Vert}{A^2-\rho^2\circ u} \leq \frac{c^{1/4}_{m}}{A} \left(\frac{1}{R}+\frac{1}{\sqrt{T}}+\sqrt{K}\right)
\end{equation*}
on $Q_{R/2,T/4,\theta}$.
Letting $\theta\to 0$,
we complete the proof of Theorem \ref{thm:gradient1}.
\end{proof}

We are now in a position to show Theorem \ref{thm:main1}.
\begin{proof}[Proof of Theorem \ref{thm:main1}]
Let $(M,g(\tau))_{\tau \in [0,\infty)}$ be backward super Ricci flow satisfying $(\ref{eq:main assumption})$ for all vector fields $V$.
Let $(N,\mathfrak{g})$ be simply connected, and $\sec \leq 0$.
Let $u:M\times [0,\infty)\to N$ be a solution to backward harmonic map heat flow.
For $R>0$
we put 
\begin{equation*}
A_{R}:=\sup_{Q_{R,R^2}}2\rho \circ u.
\end{equation*}
In view of the growth condition,
$A_{R}=o(R)$ as $R\to \infty$.
For a fixed $(x,\tau)\in M \times (0,\infty)$,
we possess $(x,\tau)\in Q_{R/2,R^2/4}$ for every sufficiently large $R>0$,
and fix such one.
From Theorem \ref{thm:gradient1} with $K=0$,
we derive
\begin{equation*}
\frac{\Vert du \Vert}{A^2_R}\leq \frac{\Vert d u \Vert}{A^2_R-\rho^2\circ u} \leq \frac{2C_m}{A_R R}
\end{equation*}
at $(x,\tau)$.
Letting $R\to \infty$,
we complete the proof of Theorem \ref{thm:main1}.
\end{proof}

One can derive the following result from the Hamilton's trace Harnack inequality (see \cite[Corollary 1.2]{H2}, and cf. \cite[Corollary 2.5]{KuS}):
\begin{cor}\label{cor:cor1}
Let $(M,g(\tau))_{\tau \in [0,\infty)}$ be a complete backward Ricci flow with bounded, non-negative curvature operator.
Let $(N,\mathfrak{g})$ be a complete, simply connected Riemannian manifold with $\sec \leq 0$.
Let $u:M\times [0,\infty)\to N$ be a solution to backward harmonic map heat flow.
If
\begin{equation*}
\rho(u(x,\tau))=o\left(\mathfrak{d}(x,\tau)+\sqrt{\tau}\right)
\end{equation*}
near infinity, then $u$ is constant.
\end{cor}

\section{Proof of Theorem \ref{thm:main2}}
We next prove Theorem \ref{thm:main2}.
The key is the following:
\begin{thm}\label{thm:gradient2}
For $K\geq 0$,
let $(M,g(\tau))_{\tau \in [0,\infty)}$ be an $m$-dimensional, admissible, complete backward $(-K)$-super Ricci flow.
We assume 
\begin{equation}\label{eq:Kmain assumption2}
\mathcal{D}(V)\geq -2K\left(H+\Vert V\Vert^2   \right),\quad \mathcal{H}(V) \geq -\frac{H}{\tau},\quad H\geq 0
\end{equation}
for all vector fields $V$.
Let $(N,\mathfrak{g})$ denote a complete Riemannian manifold with $\sec \leq \kappa$ for $\kappa>0$.
Assume that $B_{\pi/2\sqrt{\kappa}}(y_0)$ does not meet $\cut(y_0)$.
Let $u:M\times [0,\infty)\to N$ be a solution to backward harmonic map heat flow.
Suppose that the image of $u$ is contained in $B_{\pi/2\sqrt{\kappa}}(y_0)$.
For $R,T>0$,
let
\begin{equation}\label{eq:new key function}
\varphi:=1-\cos \sqrt{\kappa}\rho,\quad A:=\frac{1}{2}\left(1+\sup_{Q_{R,T}}\varphi \circ u  \right).
\end{equation}
Then there is a positive constant $C_{m}>0$ depending only on $m$ such that on $Q_{R/2,T/4}$,
\begin{equation*}
\frac{\Vert du \Vert}{A-\varphi \circ u}\leq \frac{C_{m}}{\sqrt{\kappa}} \left( \frac{1}{R}+\frac{1}{\sqrt{T}}+\sqrt{K}  \right)\sup_{Q_{R,T}}\left(\frac{1}{\cos\sqrt{\kappa}\rho \circ u}\right)^2.
\end{equation*}
\end{thm}

Unlike Theorem \ref{thm:gradient1},
this estimate seems to be new even in the context of Liouville theorems for harmonic maps (see Remark \ref{rem:new point}).

\subsection{Backward harmonic map heat flows}

Let us show the following:
\begin{lem}\label{lem:simple2}
Let $(N,\mathfrak{g})$ be $\sec \leq \kappa$ for $\kappa>0$.
Assume that $B_{\pi/2\sqrt{\kappa}}(y_0)$ does not meet $\cut(y_0)$.
Let $u:M\times [0,\infty)\to N$ be a solution to backward harmonic map heat flow.
Suppose that the image of $u$ is contained in $B_{\pi/2\sqrt{\kappa}}(y_0)$.
For $R,T>0$,
we define $\varphi$ and $A$ as \eqref{eq:new key function}.
Set
\begin{equation}\label{eq:key function2}
w:=\frac{\Vert du\Vert^2}{(A-\varphi \circ u)^2}.
\end{equation}
Then we have
\begin{align*}
(\Delta+\partial_{\tau}) w-2\frac{g(\nabla w,\nabla (\varphi \circ u))}{A-\varphi \circ u}&\geq 2\kappa(1-A)(A-\varphi \circ u)w^2\\
      &\quad +\frac{2}{(A-\varphi \circ u)^2}\sum^{m}_{i=1}\mathfrak{g}(du(\mathcal{R}(e_i)),du(e_i)).
\end{align*}
\end{lem}
\begin{proof}
By similar computations to the proof of Lemma \ref{lem:simple},
we see
\begin{align*}
(\Delta+\partial_{\tau}) w &=\frac{2g(\nabla w,\nabla (\varphi\circ u))}{A-\varphi \circ u}+\frac{2\Vert du\Vert^2~ (\Delta+\partial_{\tau}) (\varphi \circ u)}{(A-\varphi \circ u)^3}+\frac{(\Delta+\partial_{\tau}) \Vert du\Vert^2}{(A-\varphi \circ u)^2}\\
                &\quad+\frac{2g(\nabla \Vert du\Vert^2,\nabla (\varphi \circ u))}{(A-\varphi \circ u)^3}+\frac{2\Vert \nabla (\varphi \circ u) \Vert^2~\Vert du\Vert^2}{(A-\varphi \circ u)^4}.
\end{align*}
Due to the Hessian comparison,
\begin{equation*}
(\Delta+\partial_{\tau}) (\varphi \circ u)=\sum^{m}_{i=1}\nabla^2 \rho^2(du(e_i),du(e_i))\geq \kappa \cos\sqrt{\kappa}\rho \circ u \Vert du \Vert^2.
\end{equation*}
Furthermore,
Lemma \ref{lem:easy} and $\sec \leq \kappa$ lead us to
\begin{equation*}
(\Delta+\partial_{\tau}) \Vert du\Vert^2 \geq 2\Vert \nabla du\Vert^2+2\sum^{m}_{i=1}\mathfrak{g}(du(\mathcal{R}(e_i)),du(e_i))-2\kappa \Vert du \Vert^4.
\end{equation*}
It holds that
\begin{align*}
(\Delta+\partial_{\tau}) w-\frac{2g(\nabla w,\nabla (\varphi \circ u))}{A-\varphi \circ u}&\geq 2\kappa(A-\varphi \circ u)^2 \left( \frac{\cos \sqrt{\kappa}\rho }{A-\varphi \circ u}-1   \right) w^2\\
      &\quad +\frac{2}{(A-\varphi \circ u)^2}\sum^{m}_{i=1}\mathfrak{g}(du(\mathcal{R}(e_i)),du(e_i))+2\mathcal{F}\\
      &= 2\kappa(1-A)(A-\varphi \circ u)w^2\\
      &\quad +\frac{2}{(A-\varphi \circ u)^2}\sum^{m}_{i=1}\mathfrak{g}(du(\mathcal{R}(e_i)),du(e_i))+2\mathcal{F},
\end{align*}
where
\begin{equation*}
\mathcal{F}:=\frac{\Vert \nabla du\Vert^2}{(A-\varphi \circ u)^2}+\frac{\Vert \nabla (\varphi \circ u) \Vert^2~\Vert du\Vert^2}{(A-\varphi \circ u)^4}+\frac{g(\nabla \Vert du\Vert^2,\nabla (\varphi \circ u))}{(A-\varphi \circ u)^3}.
\end{equation*}
By similar computations to the proof of Lemma \ref{lem:simple},
$\mathcal{F}$ is non-negative.
\end{proof}

\subsection{Cut-off arguments}\label{sec:Cut-off arguments}

We have the following:
\begin{prop}\label{prop:difficult2}
Let $K\geq 0$.
We assume
\begin{equation*}
\mathcal{R}(V)\geq -K\Vert V \Vert^2,\quad \mathcal{D}(V)\geq -2K\left(H+\Vert V \Vert^2 \right),\quad \mathcal{H}(V) \geq -\frac{H}{\tau},\quad H\geq 0
\end{equation*}
for all vector fields $V$.
Let $(N,\mathfrak{g})$ be $\sec \leq \kappa$ for $\kappa>0$.
Assume that $B_{\pi/2\sqrt{\kappa}}(y_0)$ does not meet $\cut(y_0)$.
Let $u:M\times [0,\infty)\to N$ be a solution to backward harmonic map heat flow.
Suppose that the image of $u$ is contained in $B_{\pi/2\sqrt{\kappa}}(y_0)$.
For $R,T>0$,
we define $\varphi$ and $A$ as \eqref{eq:new key function}.
Furthermore,
we define $w$ as \eqref{eq:key function2}.
We also take a function $\psi:[0,\infty)\times [0,\infty)\to [0,1]$ in Lemma \ref{lem:cutoff} with $\alpha=3/4$,
and define $\psi$ as \eqref{eq:cutoff function}.
Then for any $\epsilon>0$, we have
\begin{align*}
2\kappa(1-A)(A-\varphi \circ u)\psi w^2
&\leq \frac{27\epsilon}{4}  \psi w^2+\frac{C^2_{3/4}}{\epsilon}\left(m^2+\frac{9}{4}+9C^2_{3/4}+  \frac{36 \kappa^2 C^2_{3/4}}{\epsilon^2}   \right)\frac{1}{R^4}\\
&\qquad \qquad \,\,\,\, +\frac{C^2}{4\epsilon}\frac{1}{T^2}+\frac{1}{\epsilon}\left(1+\frac{C^2_{3/4}}{4}\right)K^2+\Phi
\end{align*}
at every point in $Q_{R,T}$ such that the reduced distance is smooth,
where the universal constants $C_{3/4},C>0$ are given in Lemma \ref{lem:cutoff},
and put
\begin{equation}\label{eq:dimension epsilon2}
\Phi:=(\Delta+\partial_{\tau})(\psi w)-\frac{2g\left( \nabla\psi, \nabla(\psi w) \right)}{\psi}-\frac{2g(\nabla(\psi w),\nabla (\varphi \circ u))}{A-\varphi \circ u}.
\end{equation}
\end{prop}
\begin{proof}
Using Lemma \ref{lem:simple2},
we see
\begin{align*}
\Phi&=\psi \left(\Delta+\partial_{\tau} \right)w-\frac{2\psi g(\nabla w,\nabla (\varphi\circ u))}{A-\varphi \circ u}
+w\left(\Delta+\partial_{\tau}\right)\psi-\frac{2w\Vert \nabla \psi \Vert^2}{\psi}-\frac{2w g(\nabla \psi,\nabla (\varphi \circ u))}{A-\varphi \circ u}\\
    &\geq 2\kappa(1-A)(A-\varphi \circ u)\psi w^2+\frac{2\psi }{(A-\varphi \circ u)^2}\sum^{m}_{i=1}\mathfrak{g}(du(\mathcal{R}(e_i)),du(e_i))
    +w\left(\Delta+\partial_{\tau} \right)\psi\\
    &\qquad \qquad \qquad \qquad \qquad \qquad \,\,~  -\frac{2w\Vert \nabla \psi \Vert^2}{\psi}-2\frac{w g(\nabla \psi,\nabla (\varphi\circ u))}{A-\varphi \circ u}.
\end{align*}
We obtain
\begin{equation*}
2\kappa(1-A)(A-\varphi \circ u)\psi w^2 \leq \Psi_1+\Psi_2+\Psi_3+\Psi_4+\Phi
\end{equation*}
for
\begin{align*}
\Psi_1&:=-\frac{2\psi }{(A-\varphi \circ u)^2}\sum^{m}_{i=1}\mathfrak{g}(du(\mathcal{R}(e_i)),du(e_i)),\quad \Psi_2:=-w \left(\Delta+\partial_{\tau}\right)\psi,\\ \Psi_3&:=\frac{2w\Vert \nabla \psi \Vert^2}{\psi},\quad \Psi_4:=\frac{2w g(\nabla \psi,\nabla (\varphi\circ u))}{A-\varphi \circ u}.
\end{align*}
For $\Psi_1$,
the following holds:
\begin{equation*}
\Psi_1=-\frac{2\psi }{(A-\varphi \circ u)^2}\sum^{m}_{i=1}\mathfrak{g}(du(\mathcal{R}(e_i)),du(e_i)) \leq 2K \psi w\leq \epsilon \psi^2 w^2+\frac{K^2}{\epsilon}\leq \epsilon \psi w^2+\frac{K^2}{\epsilon}
\end{equation*}
in the same manner as in the proof of Proposition \ref{prop:difficult}.
For $\Psi_2,\Psi_3$,
we possess the same upper estimates as in the proof of Proposition \ref{prop:difficult}.
For $\Psi_4$,
the following holds:
\begin{align*}\label{eq:estimate of C2}
\Psi_4&=\frac{2w g(\nabla \psi,\nabla (\varphi \circ u))}{A-\varphi \circ u}\leq \frac{2w\Vert \nabla \psi \Vert \Vert \nabla (\varphi \circ u) \Vert}{A-\varphi \circ u}\leq 2 \sqrt{\kappa} w^{3/2} \Vert \nabla \psi \Vert\\ \notag
                  &\leq \frac{3 \epsilon}{4}\psi w^2+\frac{4\kappa^2}{\epsilon^3}  \frac{\Vert \nabla \psi \Vert^4}{\psi^3}\leq \frac{3 \epsilon}{4}\psi w^2+\frac{36\kappa^2C^4_{3/4}}{\epsilon^3} \,\frac{1}{R^4}.
\end{align*}
This proves the desired estimate.
\end{proof}

\subsection{Proof of Theorems \ref{thm:main2} and \ref{thm:gradient2}}\label{sec:Proof of main theorems}

We are now in a position to prove Theorem \ref{thm:gradient2}.
\begin{proof}[Proof of Theorem \ref{thm:gradient2}]
For $K\geq 0$,
let $(M,g(\tau))_{\tau \in [0,\infty)}$ be backward $(-K)$-super Ricci flow satisfying \eqref{eq:Kmain assumption2} for all vector fields $V$.
Let $(N,\mathfrak{g})$ be $\sec \leq \kappa$ for $\kappa>0$.
Assume that $B_{\pi/2\sqrt{\kappa}}(y_0)$ does not meet $\cut(y_0)$.
Let $u:M\times [0,\infty)\to N$ be a solution to backward harmonic map heat flow.
Suppose that the image of $u$ is contained in $B_{\pi/2\sqrt{\kappa}}(y_0)$.
For $R,T>0$,
we define $\varphi$ and $A$ as \eqref{eq:new key function}.
Furthermore,
we define $w$ as \eqref{eq:key function2}.
Also,
we define $\psi$ as in Proposition \ref{prop:difficult2}.
For $\theta>0$
we define $Q_{R,T,\theta}$ as \eqref{eq:compact}.
For a fixed small $\theta\in (0,T/4)$,
we take a maximum point $(\overline{x},\overline{\tau})$ of $\psi w$ in $Q_{R,T,\theta}$.
We may assume that
the reduced distance is smooth at $(\overline{x},\overline{\tau})$.

We set
\begin{equation*}
\delta:=(1-A)(A-\varphi(u(\overline{x},\overline{\tau}))).
\end{equation*}
Notice that
\begin{equation*}
\frac{1}{1-A}=2\sup_{Q_{R,T}}\frac{1}{\cos\sqrt{\kappa}\rho \circ u},\quad \frac{1}{A-\varphi(u(\overline{x},\overline{\tau}))}\leq \frac{2}{1-\sup_{Q_{R,T}  } \varphi \circ u}=2\sup_{Q_{R,T}}\frac{1}{\cos\sqrt{\kappa}\rho \circ u};
\end{equation*}
in particular,
\begin{equation*}
\frac{1}{\delta}\leq 4\sup_{Q_{R,T}}\left(\frac{1}{\cos\sqrt{\kappa}\rho \circ u}\right)^2.
\end{equation*}
Letting $\epsilon\to 4\kappa \delta/27$ in Proposition \ref{prop:difficult2}, and $\Phi(\overline{x},\overline{\tau})\leq 0$ tell us that
\begin{equation*}
\kappa \delta \psi w^2\leq \frac{27C^2_{3/4}}{4\kappa \delta}\left(m^2+\frac{9}{4}+9C^2_{3/4}+  \frac{6561 C^2_{3/4}}{4\delta^2}   \right)\frac{1}{R^4}+\frac{27C^2}{16\kappa \delta}\frac{1}{T^2}+\frac{27}{4\kappa \delta}\left(1+\frac{C^2_{3/4}}{4}\right)K^2
\end{equation*}
at $(\overline{x},\overline{\tau})$,
where $\Phi$ is defined as \eqref{eq:dimension epsilon2}.
It follows that
\begin{align*}
(\psi w)^2(\overline{x},\overline{\tau})
&\leq \frac{27C^2_{3/4}}{4\kappa^2 \delta^2}\left(m^2+\frac{9}{4}+9C^2_{3/4}+  \frac{6561 C^2_{3/4}}{4\delta^2}   \right)\frac{1}{R^4}\\
&\quad +\frac{27C^2}{16\kappa^2 \delta^2}\frac{1}{T^2}+\frac{27}{4\kappa^2 \delta^2}\left(1+\frac{C^2_{3/4}}{4}\right)K^2.
\end{align*}
Since $\delta \in (0,1)$,
there is a positive constant $\overline{c}_m>0$ depending only on $m$ such that
\begin{equation*}
(\psi w)^2(\overline{x},\overline{\tau})\leq \frac{\overline{c}_m}{\kappa^2 \delta^4} \left(\frac{1}{R^4}+\frac{1}{T^2}+K^2 \right).
\end{equation*}
Thus,
\begin{equation*}\label{eq:smooth conclusion}
(\psi w)(x,\tau) \leq (\psi w)(\overline{x},\overline{\tau})\leq \frac{\overline{c}^{1/2}_{m}}{\kappa \delta^2} \left(\frac{1}{R^2}+\frac{1}{T}+K\right)
\end{equation*}
for all $(x,\tau)\in Q_{R,T,\theta}$.
By $\psi \equiv 1$ on $Q_{R/2,T/4,\theta}$,
\begin{equation*}
\frac{\Vert d u \Vert}{A-\varphi \circ u} \leq \frac{\overline{c}^{1/4}_{m}}{\sqrt{\kappa} \delta} \left(\frac{1}{R}+\frac{1}{\sqrt{T}}+\sqrt{K}\right)
\leq \frac{4\overline{c}^{1/4}_{m}}{\sqrt{\kappa}} \left(\frac{1}{R}+\frac{1}{\sqrt{T}}+\sqrt{K}\right)\sup_{Q_{R,T}}\left(\frac{1}{\cos\sqrt{\kappa}\rho \circ u}\right)^2
\end{equation*}
on $Q_{R/2,T/4,\theta}$.
Letting $\theta\to 0$,
we complete the proof of Theorem \ref{thm:gradient2}.
\end{proof}

Let us conclude Theorem \ref{thm:main2}.
\begin{proof}[Proof of Theorem \ref{thm:main2}]
Let $(M,g(\tau))_{\tau \in [0,\infty)}$ be backward super Ricci flow satisfying \eqref{eq:main assumption2} for all vector fields $V$.
Let $(N,\mathfrak{g})$ be $\sec \leq \kappa$ for $\kappa>0$.
Assume that $B_{\pi/2\sqrt{\kappa}}(y_0)$ does not meet $\cut(y_0)$.
Let $u:M\times [0,\infty)\to N$ be a solution to backward harmonic map heat flow.
Suppose that the image of $u$ is contained in $B_{\pi/2\sqrt{\kappa}}(y_0)$.
For $R>0$
we put 
\begin{equation*}
\mathcal{A}_{R}:=\sup_{Q_{R,R^2}}\left(\frac{1}{\cos\sqrt{\kappa}\rho \circ u}\right)^2.
\end{equation*}
The growth condition says that $\mathcal{A}_{R}=o(R)$ as $R\to \infty$.
We fix $(x,\tau)\in M \times (0,\infty)$,
and a sufficiently large $R>0$.
Thanks to Theorem \ref{thm:gradient2} with $K=0$,
\begin{equation*}
\frac{\Vert du \Vert}{A}\leq \frac{\Vert d u \Vert}{A-\varphi \circ u} \leq \frac{2C_m\mathcal{A}_R}{ R}
\end{equation*}
at $(x,\tau)$,
where $\varphi$ and $A$ are defined as \eqref{eq:new key function}.
Notice that $A\leq 1$.
Thus,
by letting $R\to \infty$,
we complete the proof of Theorem \ref{thm:main2}.
\end{proof}

Similarly to Corollary \ref{cor:cor1},
we obtain the following:
\begin{cor}\label{cor:cor2}
Let $(M,g(\tau))_{\tau \in [0,\infty)}$ be a complete backward Ricci flow with bounded, non-negative curvature operator.
Let $(N,\mathfrak{g})$ be a complete Riemannian manifold with $\sec \leq \kappa$ for $\kappa>0$.
Assume that $B_{\pi/2\sqrt{\kappa}}(y_0)$ does not meet $\cut(y_0)$.
Let $u:M\times [0,\infty)\to N$ be a solution to backward harmonic map heat flow.
If the image of $u$ is contained in $B_{\pi/2\sqrt{\kappa}}(y_0)$,
and if $u$ satisfies a growth condition
\begin{equation*}
\frac{1}{\cos \sqrt{\kappa}\rho(u(x,\tau))}=o\left(\mathfrak{d}(x,\tau)^{1/2}+\tau^{1/4}\right)
\end{equation*}
near infinity, then $u$ is constant.
\end{cor}

\section{Proof of Theorem \ref{thm:improve} and Schoen-Uhlenbeck's example} 

Finally, we prove Theorem \ref{thm:improve} and compare the result with Schoen-Uhlenbeck's example. 

\subsection{Proof of Theorem \ref{thm:improve}}

In this subsection, let $(M, g)$ be an $m$-dimensional complete Riemannian manifold of non-negative Ricci curvature,
and let $(N, \mathfrak{g})$ be an $n$-dimensional complete Riemannian manifold with $\sec\leq \kappa$ for $\kappa>0$.
For harmonic maps,
we can use the following refined Kato inequality: 
\begin{lem}\label{lem:ref-kato}
For a harmonic map $u: M \to N$, we have 
\begin{equation}\label{eq:ref-kato}
\Vert\nabla \Vert du\Vert\Vert^2\leq \frac{m-1}{m}\Vert\nabla du\Vert^2.   
\end{equation}
\end{lem}
We can find the proof of this inequality for a harmonic map between spheres in the paper by Lin-Wang \cite{LiWa}. However, their computation is pointwise and only uses properties of harmonic maps. Therefore it is also valid for harmonic maps between general Riemannian manifolds. Here, we give a proof for readers' convenience. 
\begin{proof}
It is enough to show the inequality at $x_0\in M$ such that $\Vert du \Vert(x_0)\neq 0$. Let us fix such a point. We compute in normal coordinates $(x^i)=(x^1, \dots, x^m)$ around $x_0\in M$ and $(y^\alpha)=(y^1, \dots, y^n)$ around $u(x_0)\in N$. Let $u(x)=(u^1(x^1, \dots, x^m), \dots, u^n(x^1, \dots, x^m))$ be the local expression for $u:M\to N$ in these coordinates. We use the notations 	
\begin{equation*}
	u^\alpha_i=\frac{\partial u^\alpha}{\partial x^i}, \quad \text{and} \quad u^\alpha_{ij}=\frac{\partial^2u^\alpha}{\partial x^i\partial x^j}. 
\end{equation*}
Now we can write   
\begin{equation*}
	\Vert \nabla du \Vert^2(x_0)=\sum_{\alpha=1}^n \sum_{i,j=1}^m (u_{ij}^\alpha)^2(x_0) 
\end{equation*}
at $x_0\in M$. For any $1\leq \alpha\leq n$, let $\lambda^\alpha_1, \dots, \lambda^\alpha_m$ be real eigenvalues of the symmetric matrix $(u^\alpha_{ij}(x_0))$ such that $|\lambda^\alpha_1|\leq \cdots \leq |\lambda^\alpha_m|$. Then we have 
\begin{equation*}
	\Vert \nabla du \Vert^2(x_0)=\sum_{\alpha=1}^n \sum_{i=1}^m(\lambda^\alpha_i)^2. 
\end{equation*}
On the other hand, since $u:M\to N$ is a harmonic map, we have 
\begin{equation*}
	\sum_{i=1}^m u_{ii}^\alpha(x_0)=\sum_{i=1}^m \lambda_i^\alpha=0 \quad \text{for all} \quad 1\leq \alpha \leq n. 
\end{equation*} 
Using this, elementary computation yields 
\begin{equation*}
	\sum_{i=1}^{m-1}(\lambda_i^\alpha)^2\geq \frac{1}{m-1}\left(\sum_{i=1}^{m-1}\lambda_i^\alpha \right)^2=\frac{1}{m-1}(\lambda^{\alpha}_m)^2 \quad \text{for all} \quad 1\leq \alpha \leq n. 
\end{equation*}
Adding $(\lambda^\alpha_m)^2$ to the both sides of this inequality, we have 
\begin{equation*}
\Vert \nabla^2u^\alpha \Vert^2(x_0)=\sum_{i=1}^{m}(\lambda_i^\alpha)^2\geq \frac{m}{m-1}(\lambda^\alpha_m)^2 \quad \text{for all} \quad 1\leq \alpha \leq n. 
\end{equation*}
For a general $m\times m$ symmetric matrix $A$ which has real eigenvalues $\lambda_1, \dots, \lambda_m$ with $|\lambda_1|\leq \cdots \leq |\lambda_m|$, and a vector $v\in \mathbb{R}^m$, it holds that 
\begin{equation*}
	\Vert Av\Vert^2\leq |\lambda_m|^2 \Vert v\Vert^2.  
\end{equation*}
In our case, for each $1\leq \alpha \leq n$, put $v=\nabla u^\alpha(x_0)$ and $A=(u^\alpha_{ij}(x_0))$, then we have  
\begin{align*}
\Vert \nabla u^\alpha \Vert^2(x_0)\Vert \nabla^2 u^\alpha \Vert^2(x_0) &= \Vert  \nabla u^\alpha \Vert^2(x_0) \sum_{i=1}^m(\lambda^\alpha_i)^2\\ 
&\geq \frac{m}{m-1}\Vert \nabla u^\alpha \Vert^2(x_0)|\lambda^\alpha_m|^2\\
&\geq \frac{m}{m-1}\sum_{i=1}^m\left(\sum_{j=1}^m u_{ij}^\alpha(x_0) u_j^\alpha(x_0)\right)^2. 
\end{align*}
Therefore, using the Cauchy-Schwarz inequality and the Minkowski inequality, we have 
\begin{align*}
\Vert du \Vert^2(x_0) \Vert \nabla du \Vert^2(x_0)&=\left(\sum_{\alpha=1}^n \Vert \nabla u^\alpha \Vert^2(x_0)\right)\left(\sum_{\alpha=1}^n\Vert\nabla^2 u^\alpha \Vert^2(x_0)\right)\\
&\geq \left(\sum_{\alpha=1}^n\Vert\nabla u^\alpha\Vert(x_0)\Vert \nabla^2 u^\alpha \Vert(x_0) \right)^2\\
&\geq \frac{m}{m-1}\left[ \sum_{\alpha=1}^n\left\{\sum_{i=1}^m\left(\sum_{j=1}^m u_{ij}^\alpha(x_0) u_j^\alpha(x_0)\right)^2 \right\}^{\frac{1}{2}}\right]^2\\
&\geq \frac{m}{m-1}\sum_{i=1}^m\left(\sum_{\alpha=1}^n\sum_{j=1}^m u_{ij}^\alpha(x_0) u_j^\alpha(x_0)\right)^2. 
\end{align*}
Note that 
\begin{equation*}
4\Vert du \Vert^2 \Vert \nabla\Vert du\Vert \Vert^2 = \Vert \nabla \Vert du \Vert^2 \Vert^2=4\sum_{i=1}^m\left(\sum_{\alpha=1}^n\sum_{j=1}^m u_{ij}^\alpha u_j^\alpha\right)^2. 
\end{equation*}
Hence we obtain 
\begin{equation*}
\Vert du \Vert^2(x_0) \Vert \nabla du \Vert^2(x_0)\geq \frac{m}{m-1}\Vert du \Vert^2(x_0) \Vert \nabla\Vert du\Vert \Vert^2(x_0). 
\end{equation*}
This completes the proof of Lemma \ref{lem:ref-kato}. 
\end{proof}

Now we are in a position to prove Theorem \ref{thm:improve}. We use the technique for minimal hypersurfaces developed by Ecker-Huisken in \cite{EH}.  
\begin{proof}[Proof of Theorem \ref{thm:improve}]
The Bochner formula for a harmonic map $u:M\to N$ (i.e., harmonic map version of Lemma \ref{lem:easy}) combined with the refined Kato inequality \eqref{eq:ref-kato} and curvature assumptions on $M, N$ tells us that 
\begin{equation}\label{eq:bochner-kato}
\Delta\Vert du\Vert^2\geq -2\kappa \Vert du\Vert^4+2\Vert\nabla du\Vert^2 \geq -2\kappa \Vert du\Vert^4+\frac{2m}{m-1}\Vert\nabla \Vert du\Vert\Vert^2. 
\end{equation}
Let $v(x)=1/\cos\sqrt{\kappa}\rho(u(x))$. Note that this is well-defined when $u(M)\subset B_{\pi/2\sqrt{\kappa}}(y_0)$. The Hessian comparison theorem under the curvature assumption on $N$ implies 
\begin{equation}\label{eq:lap-v}
	\Delta v=v^2\Delta(\varphi\circ u)+\frac{2\Vert\nabla v\Vert^2}{v}\geq \kappa \Vert du\Vert^2v+\frac{2\Vert\nabla v\Vert^2}{v}, 
\end{equation}
where $\varphi=1-\cos\sqrt \kappa\rho$. Using \eqref{eq:bochner-kato} and \eqref{eq:lap-v}, a direct computation yields 
\begin{align*}
\Delta(\Vert du\Vert^p v^q)&\geq \kappa(q-p)\Vert du\Vert^{p+2}v^q\\
	&+p\left(p-1+\frac{1}{m-1}\right)\Vert du\Vert^{p-2}v^q\Vert\nabla\Vert du\Vert\Vert^2+q(q+1)\Vert du\Vert^p v^{q-2}\Vert\nabla v\Vert^2\\&+2pq\Vert du\Vert^{p-1}v^{q-1}g(\nabla\Vert du\Vert, \nabla v), 
\end{align*}
where $p, q$ are determined later. 
Using the Cauchy-Schwarz inequality and the Young inequality (with $\varepsilon>0$) for the last term, we have 
\begin{align}\label{eq:main-ineq}
\Delta(\Vert du\Vert^p v^q)&\geq \kappa(q-p)\Vert du\Vert^{p+2}v^q \\
	&\quad +p\left(p-1+\frac{1}{m-1}-\varepsilon q \right)\Vert du\Vert^{p-2}v^q
	\Vert \nabla\Vert du\Vert \Vert^2  \nonumber \\
	&\quad +q\left(q+1-\varepsilon^{-1}p\right)\Vert du\Vert^p v^{q-2}\Vert\nabla v\Vert^2. \nonumber
\end{align}
For given $m\geq 2$, let $q=p>2m-3$ and $\varepsilon=q/(q+1)$. Then we have 
\begin{equation*}
\Delta(\Vert du\Vert^q v^q)\geq 0,  
\end{equation*}
i.e., $\Vert du\Vert^q v^q$ is a subharmonic function on $M$. Therefore, we can use Li-Schoen's mean value inequality (see e.g., \cite[Theorem 7.2]{PLi}) to conclude 
\begin{equation}\label{eq:mean value ineq}
\sup_{B_{R/4}(x_0)}\Vert du\Vert^{2q} v^{2q}\leq \frac{C_m}{\vol({B_{R}(x_0))}}\int_{B_{R}(x_0)}\Vert du\Vert^{2q}v^{2q},  
\end{equation} 
where $C_m$ is a positive constant depending only on $m$.

On the other hand, choosing $q=p+1>2m-3$ and $\varepsilon=(q-1)/(q+1)$ in \eqref{eq:main-ineq}, we have  
\begin{equation*}
\Delta(\Vert du\Vert^{q-1}v^q) \geq \kappa \Vert du\Vert^{q+1}v^q. 
\end{equation*}
We multiply this by $\Vert du\Vert^{q-1}v^q\eta^{2q}$, where $\eta$ is a test function on $M$ with compact support, and then integrating by parts with the Cauchy-Schwarz inequality and the Young inequality yields 
\begin{align*}
\kappa\int_{M}\Vert du\Vert^{2q}v^{2q}\eta^{2q}\leq \frac{1}{2}\int_{M}\Vert du\Vert^{2q-2}v^{2q}\eta^{2q-2}\Vert\nabla \eta\Vert^2.  
\end{align*}
Recall the generalized Young inequality for $a, b\geq 0$ with arbitrary $\varepsilon>0$:
\begin{equation*}
ab\leq \varepsilon\left(\frac{q-1}{q}\right)a^{q/(q-1)}+\frac{\varepsilon^{-(q-1)}}{q}b^q.  
\end{equation*}
Putting $a=\Vert du \Vert^{2q-2}\eta^{2q-2}$ and $b=\Vert \nabla\eta\Vert^2$ in this inequality, we have 
\begin{equation*}
\left(\kappa-\frac{\varepsilon(q-1)}{2q}\right)\int_{M}\Vert du\Vert^{2q}v^{2q}\eta^{2q}\leq \frac{\varepsilon^{-(q-1)}}{2q}\int_{M}v^{2q}\Vert \nabla \eta\Vert^{2q}. 
\end{equation*}
We take $\varepsilon=\kappa q/(q-1)$ to get 
\begin{equation*}
	\int_{M}\Vert du \Vert^{2q}v^{2q}\eta^{2q}\leq \frac{(q-1)^{q-1}}{\kappa^q q^q}\int_{M}v^{2q}\Vert \nabla \eta\Vert^{2q}. 
\end{equation*} 
Choosing $\eta$ as the standard cut-off function in this inequality, we obtain 
\begin{align}\label{eq:ineq-cutoff}
\int_{B_{R}(x_0)}\Vert du\Vert^{2q}v^{2q}&\leq \frac{(q-1)^{q-1}}{\kappa^q q^q R^{2q}} \int_{B_{2R}(x_0)}v^{2q} \\ \notag
&\leq \frac{(q-1)^{q-1}}{\kappa^q q^q R^{2q}}\vol(B_{2R}(x_0))\sup_{B_{2R}(x_0)}v^{2q}.
\end{align} 
Combining \eqref{eq:mean value ineq}, \eqref{eq:ineq-cutoff} and the Bishop-Gromov volume comparison, it follows that \begin{equation*}
\sup_{B_{R/4}(x_0)}\Vert du\Vert v\leq \left(\frac{C_m(q-1)^{q-1}}{\kappa^q q^q R^{2q}}\right)^{1/2q} \left( \frac{\vol(B_{2R}(x_0))}{\vol(B_{R}(x_0))} \right)^{1/2q}\sup_{B_{2R}(x_0)}v\leq  C_{m, \kappa, q}\frac{o(R)}{R}, 
\end{equation*}
where $C_{m,\kappa,q}$ is a positive constant depending only on $m,\,\kappa$ and $q$.
We here notice that $q$ depends only on $m$. 
Letting $R\to \infty$, we complete the proof of Theorem \ref{thm:improve}. 
\end{proof}

\subsection{Schoen-Uhlenbeck's radial solution}\label{subsec:SU-example}
We examine our growth condition \eqref{eq:improve} in Theorem \ref{thm:improve} by comparing with the known example.
In \cite[Example 2.2, Corollary 2.6]{SU}, Schoen-Uhlenbeck showed that a smooth harmonic map $u:\mathbb{R}^m \to \mathbb{S}^n_+$ is necessarily constant for $m\leq 6$, and for $m\geq 7$ such a map exists as a radial solution. 

Now we consider a radial solution, that is, a harmonic map $u:\mathbb{R}^m \to \mathbb{S}^m\subset \mathbb{R}^{m+1}$ of the form $u(r, \theta)=(\rho(r), \theta)$, where $(r, \theta)=(d(x), \theta)$ are polar coordinates in $\mathbb{R}^m$ and $(\rho, \theta)$ are polar coordinates in $\mathbb{S}^{m}$ centered at the north pole. Then the harmonic map equation can be reduced to the following second order nonlinear ODE of $\rho(r)$: 
\begin{equation}\label{eq:ode1}
\frac{d^2\rho}{dr^2}+\frac{m-1}{r}\frac{d\rho}{dr}-\frac{m-1}{2r^2}\sin(2\rho)=0
\end{equation}
for $0<r<\infty$ with initial conditions 
\begin{equation*}
\lim_{r\to 0}\rho(r)=0, \quad \lim_{r\to 0}\frac{d\rho}{dr}(r)>0. 
\end{equation*}
According to Schoen-Uhlenbeck \cite{SU}, if {$m\geq 7$}, $\rho(r)$ lies below the line $\rho=\pi/2$, is increasing and asymptotic to $\pi/2$. As a consequence, we have  
\begin{equation*}
\frac{1}{\cos(\rho(r))} \to \infty 
\end{equation*}
as $r\to \infty$. Now we want to know the precise growth order of this near infinity. Following Schoen-Uhlenbeck \cite{SU}, it is convenient to make the change of variables 
\begin{equation*}
\alpha=2\rho, \quad t=\log r. 
\end{equation*}
Then ODE \eqref{eq:ode1} becomes 
\begin{equation}\label{eq:ode2}
\frac{d^2\alpha}{dt^2}+(m-2)\frac{d\alpha}{dt}-(m-1)\sin \alpha=0 
\end{equation}
for $-\infty<t<\infty$ with 
\begin{equation*}
\lim_{t\to -\infty}\alpha(t)=0, \quad \lim_{t\to -\infty}\frac{d\alpha}{dt}(t)=0. 
\end{equation*}
This is the nonlinear damped pendulum differential equation. Introducing $\beta=d\alpha/dt$ we get the first order autonomous system 
\begin{equation}\label{eq:ode3}
\frac{d\alpha}{dt}=\beta, \quad \frac{d\beta}{dt}=(2-m)\beta+(m-1)\sin \alpha. 
\end{equation}
A standard way to analyze the behavior of nonlinear ODE near a critical point is to study the linearized equation at the point. In our case, we consider the linearization of the system \eqref{eq:ode3} at the critical point $(\alpha, \beta)=(\pi, 0)$: 
\begin{equation*}
\frac{d\tilde\alpha}{dt}=\tilde\beta, \quad \frac{d\tilde\beta}{dt}=(2-m)\tilde\beta+(m-1)(\pi-\tilde\alpha). 
\end{equation*}
or equivalently, 
\begin{equation}\label{eq:ode4}
\frac{d^2\tilde\alpha}{dt^2}+(m-2)\frac{d\tilde\alpha}{dt}+(m-1)\tilde\alpha=(m-1)\pi. 
\end{equation}
The characteristic equation is $\lambda^2+(m-2)\lambda+(m-1)=0$ and its roots are 
\begin{equation*}
\lambda_1(m)=\frac{-(m-2)+ \sqrt{m^2-8m+8}}{2}, \quad \lambda_2(m)=\frac{-(m-2)- \sqrt{m^2-8m+8}}{2}. 
\end{equation*}
If $m\geq 7$, $\lambda_1(m)$ and $\lambda_2(m)$ are both negative real. In this case, it is known that the corresponding critical point $(\pi, 0)$ of the original nonlinear autonomous system \eqref{eq:ode2} is asymptotically stable node. The general solution of the linearized ODE \eqref{eq:ode4} is given by 
\begin{equation*}
\tilde\alpha(t)=\pi+C_1e^{\lambda_1(m)t}+C_2e^{\lambda_2(m)t},  
\end{equation*}
where $C_1$ and $C_2$ are arbitrary constants. Putting $\tilde\rho=2\tilde\alpha$ and $t=\log{r}$, we have 
\begin{equation*}
\frac{\pi}{2}-\tilde\rho(r)=C_1r^{-N_1}+C_2r^{-N_2}=\frac{C_1r^{N_2-N_1}+C_2}{r^{N_2}}, 
\end{equation*}
where $0<N_1:=-\lambda_1(m)<-\lambda_2(m)=:N_2$. Therefore, 
\begin{equation}\label{SU-order1}
\frac{1}{\cos(\tilde\rho(r))}=\frac{1}{\sin\left(\frac{\pi}{2}-\tilde\rho(r)\right)}\sim\frac{r^{N_1}}{C_1}-\frac{C_2 r^{N_1}}{C_1(C_1 r^{N_2-N_1}+C_2)}  
\end{equation}
as $r\to \infty$. Note that 
\begin{equation}\label{SU-order2}
N_1(m)\searrow 1 \quad \text{and} \quad N_2(m) \nearrow \infty \quad \text{as} \quad m \to \infty. 
\end{equation}
Since the solution $\rho(r)$ of the original nonlinear ODE \eqref{eq:ode1} is approximated by the linearized one $\tilde{\rho}(r)$ as $r\to \infty$, $\rho(r)$ does not satisfy the growth condition \eqref{eq:improve} in our Liouville theorem. In addition, \eqref{SU-order1} and \eqref{SU-order2} tell us that our growth condition \eqref{eq:improve} is almost sharp. 


\subsection*{{\rm Acknowledgements}} 
The authors wish to express their gratitude to Toru Kajigaya for letting us know the example by Schoen-Uhlenbeck. The first author was supported by JSPS KAKENHI (JP19K14521). 
The second author was supported by JSPS Grant-in-Aid for Scientific Research on Innovative Areas ``Discrete Geometric Analysis for Materials Design" (17H06460).


\end{document}